\documentclass{amsart}

\usepackage{amsmath,amsthm,mathrsfs,amssymb,graphicx}
\usepackage[all]{xy}
\usepackage[unicode=true,
					 pdfusetitle=true,
					 bookmarks=true,
					 bookmarksnumbered=false,
					 bookmarksopen=false,
					 breaklinks=true,pdfborder={0 0 0},
					 backref=false,
					 colorlinks=true,
					 plainpages=false,
					 pdfpagelabels,
					 linktocpage=true]{hyperref}
 \numberwithin{equation}{section}
\usepackage{tikz}

\newtheorem{theorem}{Theorem}[section]
\newtheorem{corollary}[theorem]{Corollary}

\newtheorem{proposition}[theorem]{Proposition}
\newtheorem{claim}[theorem]{Claim}
\theoremstyle{definition}
\newtheorem{definition}[theorem]{Definition}
\newtheorem{question}[theorem]{Question}
\newtheorem{remark}[theorem]{Remark}

\newtheoremstyle{principle}{}{}{\itshape}{}{\bfseries}{.}{.5em}{\thmnote{#3}#1}
\theoremstyle{principle}
\newtheorem*{principle}{}
\newtheoremstyle{case}{}{}{}{}{\itshape}{.}{.5em}{\thmnote{Case #3}#1}
\theoremstyle{case}

\newcommand{\N}{\mathbb{N}}

\newcommand{\res}{\mathbin{\upharpoonright}}

\newcommand{\sequence}[1]{\langle #1 \rangle}

\newcommand{\Iff}{\Longleftrightarrow}

\newcommand{\0}{\mathbf{0}}

\newcommand{\ZF}{\mathsf{ZF}}

\newcommand{\RCA}{\mathsf{RCA}_0}
\newcommand{\ACA}{\mathsf{ACA}_0}
\newcommand{\WKL}{\mathsf{WKL}_0}

\newcommand{\I}{\mathsf{I}}

\newcommand{\IP}{\mathsf{IP}}

\newcommand{\AMT}{\mathsf{AMT}}
\newcommand{\OPT}{\mathsf{OPT}}
\newcommand{\HYP}{\mathsf{HYP}}

\begin{document}

% For article class
%\title{Reverse mathematics and the finite intersection principle}
%\author{Damir D. Dzhafarov\\Department of Mathematics\\University of Notre Dame
%\and
%Carl Mummert\\Department of Mathematics\\Marshall University}
\date{September 15, 2011}

% For amsart class
\title{On the strength of the finite intersection principle}
\author{Damir D. Dzhafarov}
\address{Department of Mathematics\\University of Notre Dame\\255 Hurley Hall\\Notre Dame, Indiana 46556 U.S.A.}
\email{ddzhafar@nd.edu}
\author{Carl Mummert}
\address{Department of Mathematics\\Marshall University\\1 John Marshall Drive\\Huntington, West Virginia 25755 U.S.A.} \email{mummertc@marshall.edu}
\thanks{The authors are grateful to Denis Hirschfeldt, Antonio Montalb\'{a}n, and Robert Soare for valuable comments and suggestions. The first author was partially supported by an NSF Graduate Research Fellowship and an NSF Postdoctoral Fellowship.}

\begin{abstract}
We study the logical content of several maximality principles related to the finite intersection principle ($F\IP$) in set theory. Classically, these are all equivalent to the axiom of choice, but in the context of reverse mathematics their strengths vary: some are equivalent to $\ACA$ over $\RCA$, while others are strictly weaker, and incomparable with $\WKL$. We show that there is a computable instance of $F\IP$ all of whose solutions have hyperimmune degree, and that every computable instance has a solution in every nonzero c.e.\ degree. In terms of other weak principles previously studied in the literature, the fomer result translates to $F\IP$ implying the omitting partial types principle~($\mathsf{OPT}$). We also show that, modulo $\Sigma^0_2$ induction, $F\IP$ lies strictly below the atomic model theorem ($\mathsf{AMT}$).
\end{abstract}

\maketitle

\section{Introduction}

After Zermelo introduced the axiom of choice in 1904, set theorists began to obtain results proving other set-theoretic principles equivalent to it (relative to choice-free axiomatizations of set theory, such as $\ZF$). These equivalence results, and their further development, now constitute a program in set theory, which has been documented in detail by Jech~\cite{Jech-1973} and by Rubin and Rubin~\cite{RR-1970, RR-1985}.  Moore~\cite{Moore-1982} provides a general historical account of the axiom of choice.

In this article, we study the logical content of several such equivalences from the point of view of computability theory and reverse mathematics. Specifically, we focus on maximality principles related to the following:

\begin{principle}[Finite intersection principle]
Every family of sets has a $\subseteq$-maximal subfamily with the finite intersection property;
\end{principle}

\noindent This  research has two closely related motivations. First, we wish to study various equivalents of the axiom of choice to determine how they compare with one another and with other mathematical principles, in the spirit of the program of reverse mathematics. 
 This program is devoted to gauging the relative strengths of (countable analogues of) mathematical theorems by calibrating the precise set existence axioms necessary and sufficient to carry out their proofs in second-order arithmetic. Second, we wish to explore potential new connections between set-theoretic principles and computability-theoretic constructions, such as have emerged in the investigations of other theorems, looking for new insights into the underlying combinatorics of the principles. (For examples, see Hirschfeldt and Shore \cite[Section 1]{HS-2007}, and also Section \ref{sec_FIPandotherprinc} below.) We refer to Soare \cite{Soare-1987} and Simpson \cite{Simpson-2009}, respectively, for general background in computability theory and reverse mathematics.

Various forms of the axiom of choice have been studied in the present context, including direct formalizations of choice principles in second-order arithmetic by Simpson~\cite[Section VII.6]{Simpson-2009}; countable well-orderings by Friedman and Hirst~\cite{HF-1990} and Hirst~\cite{Hirst-2005}; and principles related to properties of finite character by Dzhafarov and Mummert \cite{DM-TA}. These principles display varying strengths, but tend to be at least as strong as $\ACA$. By contrast, the finite intersection principle and its variants will turn out to be strictly weaker than $\ACA$ and incomparable with $\WKL$. We establish a link between maximal subfamilies with the finite intersection property and sets of hyperimmune degree, which allows us to closely locate the positions of these principles among the statements lying between $\RCA$ and $\ACA$. In particular, we show that they are closely related in strength to the atomic model theorem, studied by Hirschfeldt, Shore, and Slaman \cite{HSS-2009}.

We pass to the formal definitions needed for the sequel.

\begin{definition}\label{definition_family}\
\begin{enumerate}
\item We define a \textit{family of sets} to be a sequence $A = \langle A_i : i \in \omega \rangle$ of sets. A family $A$ is \emph{nontrivial} if $A_i \neq \emptyset$ for some $i \in \omega$. 
\item Given a family of sets $A$, we say a set $X$ is \emph{in} $A$, and write $X \in A$, if $X = A_i$ for some $i \in \omega$. A family of sets $B$ is a \emph{subfamily} of $A$ if every set in $B$ is in~$A$, that is, $(\forall i)(\exists j)[B_i = A_j]$.
\item Two sets $A_i, A_j \in A$ are \textit{distinct} if they differ extensionally as sets.
\end{enumerate}
\end{definition}

Our definition of a subfamily is intentionally weak; see 
Proposition~\ref{prop_strongerfamACA_0} below and the remarks preceding it. 

\begin{definition}\label{definition_intproperties} 
Let $A = \langle A_i : i \in \omega \rangle$ be a family of sets and fix $n \geq 2$. Then $A$ has the
\begin{itemize}
\item \emph{$D_n$ intersection property} if the intersection of any $n$ distinct sets in $A$ is empty;
\item \emph{$\overline{D}_n$ intersection property} if the intersection of any $n$ distinct sets in $A$ is nonempty;
\item \emph{F intersection property} if for every $m \geq 2$, the intersection of any $m$ distinct sets in $A$ is nonempty.
\end{itemize}
\end{definition}

\begin{definition}\label{definition_maximalfam}
Let $A$ be a family of sets, let $P$ be any of the properties in Definition~\ref{definition_intproperties}, and let $B$ be a subfamily of $A$ with the $P$ intersection property. We say $B$ is a \emph{maximal} such subfamily if for every other such subfamily $C$, $B$ being a subfamily of $C$ implies $C$ is a subfamily of $B$.
\end{definition}

\noindent It is straightforward to formalize Definitions~\ref{definition_family}--\ref{definition_maximalfam} in $\RCA$.

Given a family $A = \langle A_i : i \in \omega \rangle$ and some $J \in \omega^{\omega}$, we use the notation $\langle A_{J(i)} : i \in \omega \rangle$ for the subfamily $\langle B_i : i \in \omega \rangle$ where $B_i = A_{J(i)}$. We call this the subfamily \emph{defined} by~$J$. Given a finite set $\{j_0,\ldots,j_{n-1}\} \subset \omega$, we let $\langle A_{j_0},\ldots,A_{j_{n-1}} \rangle$ denote the subfamily $\langle B_i : i \in \N \rangle$ where $B_i = A_{j_i}$ for $i < n$ and $B_i = A_{j_{n-1}}$ for $i \geq n$. More generally, we call a subfamily $B$ of $A$ \emph{finite} if there only finitely many distinct $A_i$ in $B$.

Let $P$ be any of the properties in Definition~\ref{definition_intproperties}. We shall be interested in the following maximality principles:

\begin{principle}[\textit{P} intersection principle $(P\IP)$]
Every nontrivial family of sets has a maximal subfamily with the $P$ intersection property.
\end{principle}

\noindent Following common usage, we shall refer to a given family as an \emph{instance} of $P\IP$, and to a maximal subfamily with the $P$ intersection property as a \emph{solution} to this instance.

The classic set-theoretic analogues of $D_n\IP$ and $\overline{D}_n\IP$ in the catalogue of Rubin and Rubin \cite{RR-1985} of equivalences of the axiom of choice are $\mathsf{M}\, 8 \, (D_n)$ and $\mathsf{M}\, 8 \, (\overline{D}_n)$, respectively; of $F\IP$ it is $\mathsf{M} \,14$. For additional references and results concerning these forms, see \cite[pp.~54--56,~60]{RR-1985}.

\begin{remark}\label{rem_FIPsubfamilies} 
Although we do not make it an explicit part of the definition, all of the families $\langle A_i : i \in \omega \rangle$ we construct in our results will have the property that for each $i$, $A_i$ contains $2i$ and otherwise contains only odd numbers. This will have the advantage that if we are given an arbitrary subfamily $B = \langle B_i : i \in \omega \rangle$ of some such family, we can, for each $i$, uniformly $B$-computably find the (unique) $j$ such that $B_i = A_j$. If $A$ is computable, each subfamily $B$ will then be of the form $\langle A_{J(j)} : j \in \omega \rangle$ for some $J \colon \omega \to \omega$ with $J \equiv_T B$.
\end{remark}

\section{\texorpdfstring{Basic implications and equivalences to $\ACA$}{Basic implications and equivalents to ACA0}}\label{subsec_FIPimplications} 
The following pair of propositions establishes the basic relations that hold among the principles we have defined. The proof of the first is straightforward and so we omit it.

\begin{proposition}\label{prop_PIPinACA0} 
For any property $P$ in Definition~\ref{definition_intproperties}, $P\IP$ is provable in $\ACA$.
\end{proposition}

\begin{proposition}
For each standard $n \geq 2$, the following are provable in $\RCA$:
\begin{enumerate}
\item $F\IP$ implies $\overline{D}_n\IP$;
\item $\overline{D}_{n+1}\IP$ implies $\overline{D}_n\IP$.
\end{enumerate}
\end{proposition}

\begin{proof}
To prove (1), let $A = \langle A_i : i \in \N \rangle$ be a nontrivial family of sets. By recursion, define a new family $\widehat{A} = \langle \widehat{A}_i : i \in \N \rangle$ with the property that for every finite set $F$ with $|F| \geq n$,
\begin{equation}\label{property_int}
\bigcap_{i \in F} \widehat{A}_i \neq \emptyset \Iff (\forall G \subseteq F)[\,|G| = n \implies \bigcap_{i \in G} A_i \neq \emptyset \,].
\end{equation}
First, for all $i \neq j$, put $2i \in \widehat{A}_i$ and $2j \notin \widehat{A}_i$. Then, define the $\widehat{A}_i$ on successively longer initial segments of the odd numbers. Suppose that the $\widehat{A}_i$ have been defined precisely on the odd numbers less than $2s+1$. Consider all finite sets $F \subseteq \{0,\ldots,s\}$ such that $|F| \geq n$ and for every $G \subseteq F$ with $|G| = n$ there is an $x \leq s$ belonging to $\bigcap_{i \in G} A_i$. If no such $F$ exists, enumerate $2s+1$ into the complement of $\widehat{A}_i$ for all~$i\in \N$. Otherwise, list these sets as $F_0,\ldots,F_{k-1}$, and for each $j < k$, enumerate $ 2(s+j)+1$ into $\widehat{A}_i$ if $i \in F_j$, and into the complement of $\widehat{A}_i$ if $i \notin F_j$. Thus, $\bigcap_{i \in F_j} \widehat{A}_i \neq \emptyset$, as desired.

The family $\widehat{A}$ exists by $\Delta^0_1$ comprehension, and is nontrivial by construction. It is also easily seen to satisfy \eqref{property_int}. Applying $F\IP$, let $\widehat{B} = \langle \widehat{B}_i : i \in \N \rangle$ be a maximal subfamily of $\widehat{A}$ with the $F$ intersection property. As each $\widehat{B}_i$ contains $2j$ for the $j$ such that $\widehat{B}_i = \widehat{A}_j$, and otherwise contains only odd numbers, it follows (by formalizing Remark \ref{rem_FIPsubfamilies}) that there is a $J \colon \omega \to \omega$ such that $\widehat{B} = \sequence{\widehat{A}_{J(j)} : j \in \N}$. We claim that $B = \sequence{A_{J(j)} : j \in \N}$ is a maximal subfamily of $A$ with the $\overline{D}_n$ intersection property. This follows from \eqref{property_int}, with the fact that $B$ has this property being clear. To show maximality, suppose $A_i \notin B$. Then $\widehat{A}_i \notin \widehat{B}$, so since $\widehat{B}$ is maximal, there must exist a finite set $F$ containing $i$ and otherwise only members of the range of $J$ such that $\bigcap_{j \in F} A_j = \emptyset$. By adding elements to $F$ if necessary, we may assume $|F| \geq n$, whence the definition of $\widehat{A}$ implies there is a $G \subseteq F$ such that $|G| = n$ and $\bigcap_{j \in G} A_{j} = \emptyset$. Now $G$ must contain $i$, since otherwise each $A_j$ for $j \in G$ would be in $B$ and $\bigcap_{j \in G} A_{j}$ could not be empty. It follows that no subfamily of $A$ with the $\overline{D}_n$ intersection property can contain $A_i$ in addition to all $A_j \in B$, as desired.

A similar argument can be used to prove (2). The construction of $\widehat{A}$ is simply modified so that, instead of looking at finite sets $F \subseteq\{0,\ldots,s\}$ with $|F| \geq n$, it considers those with $|F| = n+1$.
\end{proof}

We do not know whether the implications from $F\IP$ to $\overline{D}_n\IP$ or from $\overline{D}_{n+1}\IP$ to $\overline{D}_n\IP$ are strict. However, all of our results in the sequel hold equally well for $F\IP$ as they do for $\overline{D}_2\IP$. Thus, we shall formulate all implications over $\RCA$ involving these principles as being to $F\IP$ and from $\overline{D}_2\IP$.

An apparent weakness of our definition of subfamily is that we cannot, in general, effectively decide which members of a family are in a given subfamily. The following proposition demonstrates that if the definition were strengthened to make this decidable, all the intersection principles would collapse to $\ACA$. The subsequent proposition shows that this happens for $P = D_n$ even with the weak definition.

\begin{proposition}\label{prop_strongerfamACA_0}
Let $P$ be any of the properties in Definition~\ref{definition_intproperties}. The following are equivalent over $\RCA$:
\begin{enumerate}
\item $\ACA$;
\item every nontrivial family of sets $\langle A_i : i \in \N \rangle$ has a maximal subfamily $B$ with the $P$ intersection property, and the set $I = \{i \in \N: A_i \in B\}$ exists.
\end{enumerate}
\end{proposition}

\begin{proof} 
That~(1) implies~(2) is proved similarly to Proposition~\ref{prop_PIPinACA0}.

To show that (2) implies (1), we work in $\RCA$ and let $f\colon \N \to \N$ be a given function. For each $i$, let
\[
A_i = \{2i\} \cup \{2a+1 : (\exists b \leq a)[f(b) =i]\}.
\]
noting that $i \in \operatorname{range}(f)$ if and only if $A_i$ is not a singleton, in which case $A_i$ contains cofinitely many odd numbers. Consequently, for every finite $F \subset \N$ with $|F| \geq 2$, we have $\bigcap_{i \in F} A_i \neq \emptyset$ if and only if each $i \in F$ is in the range of~$f$.

Apply (2) with to the family $A = \langle A_i : i \in \N \rangle$ to find the corresponding subfamily $B$ and set~$I$. If $P = D_n$ then there are at most~$n-1$ distinct~$j$ such that $j \in \operatorname{range}(f)$ and $A_j \in B$. For each~$i$ not among these $j$ we have
\[
i \in \operatorname{range}(f) \Iff A_i \notin B \Iff i \notin I.
\]
If instead $P = F$ or $P = \overline{D}_n$ then each~$B_i$ contains cofinitely many odd numbers and we have
\[ 
i \in \operatorname{range}(f) \Iff A_i \in B \Iff i \in I.
\]
In any case, the range of $f$ exists.
\end{proof}

\begin{proposition}\label{prop_DnIPtoACA0} 
For each standard $n \geq 2$, $D_n\IP$ is equivalent to $\ACA$ over $\RCA$.
\end{proposition}

\begin{proof} 
Fix a function $f \colon \N \to \N$, and let $A$ be the family defined in the preceding proposition. Let $B = \langle B_i : i \in \N \rangle$ be the family obtained from applying $D_n\IP$ to~$A$. As above, there can be at most $n-1$ many $j$ such that $j \in \operatorname{range}(f)$ and $A_j \in B$. For~$i$ not equal to any such~$j$, we have
\[ 
i \in \operatorname{range}(f) \Iff A_i \notin B \Iff (\forall k)[2i \notin B_k],
\]
giving us a $\Pi^0_1$ definition of the range of~$f$. Since the range is also definable by a $\Sigma^0_1$ formula, we conclude by $\Delta^0_1$ comprehension that it exists.
\end{proof}

We conclude this section by showing that, by contrast, $F\IP$ is strictly weaker than~$\ACA$. We will obtain a considerable strengthening of this fact in Theorem \ref{thm_FIPpermitting}, but the proof here further illustrates the flexibility of our definition of subfamily. 

\begin{proposition}\label{prop_FIPnotrevACA0}
Every computable nontrivial family has a low maximal subfamily with the $F$ intersection property.
\end{proposition}

\begin{proof} 
Given $A = \langle A_i : i \in \omega \rangle$ computable and nontrivial, consider the notion of forcing whose conditions are strings $\sigma \in \omega^{<\omega}$ such that some number $\leq \sigma(|\sigma|-1)$ belongs to $A_{\sigma(i)}$ for all $i < |\sigma| - 1$, and $\tau \leq\sigma$ if $\tau \res |\tau| - 1 \succeq \sigma \res |\sigma| - 1$. Now fix any $A_i \neq \emptyset$, say with $a \in A_i$, and let $\sigma_0 = i a$. Given $\sigma_{2e}$ for some $e \in \omega$, ask if there is a condition $\sigma \leq \sigma_{2e}$ such that $\Phi^{\sigma \res |\sigma| - 1}_e(e)\downarrow$. If so, let $\sigma_{2e+1}$ be the least such $\sigma$ of length greater than $|\sigma_{2e}|$, and if not, let $\sigma_{2e+1} = \sigma_{2e}$. Given $\sigma_{2e+1}$, ask if there is a condition $\sigma \leq \sigma_{2e+1}$ such that $\sigma(i) = e$ for some $i < |\sigma| - 1$. If so, let $\sigma_{2e+2}$ be the least such $\sigma$, and if not, let $\sigma_{2e+2} = \sigma_{2e+1}$. A standard argument establishes that $J = \bigcup_{e \in \omega} \left( \sigma_e \res |\sigma_e| - 1\right )$ is low, and hence that so is $B = \langle A_{J(i)} : i \in \omega \rangle$. It is clear that $B$ is a maximal subfamily of $A$ with the $F$ intersection property.
\end{proof}

Iterating and dovetailing this argument produces an $\omega$-model witnessing:

\begin{corollary}\label{cor_FIPnotrevACA0}
Over $\RCA$, $F\IP$ does not imply $\ACA$.
\end{corollary}

\section{Connections with hyperimmunity}\label{sec_hyperimmunity}

Corollary \ref{cor_FIPnotrevACA0} naturally leads to the question of whether $F\IP$ (or any one of the principles $\overline{D}_n\IP$) is provable in $\RCA$, or at least in $\WKL$. We show in this section that the answer to both questions is no. Recall that a Turing degree is \emph{hyperimmune} if it bounds the degree of a function not dominated by any computable function; a degree which is not hyperimmune is \emph{hyperimmune-free}. In this section, we prove the following result:

\begin{theorem}\label{thm_FIPtoOPT}
There is a computable nontrivial family of sets, any maximal subfamily of which with the $\overline{D}_2$ intersection property has hyperimmune degree.
\end{theorem}

\noindent As there is an $\omega$-model of $\WKL$ consisting entirely of sets of hyperimmune-free degree, this yields:

\begin{corollary}\label{cor_WKLnottoFIP}
The principle $\overline{D}_2\IP$ is not provable in $\WKL$.
\end{corollary}

In the proof of the theorem, we build a computable family $A = \langle A_i : i \in \omega \rangle$ by stages, letting $A_{i,s}$ be the set of elements enumerated into $A_i$ by stage $s$, which will always be finite. As usual, we initially put $2i$ into $A_i$ for every $i$, and then put in only odd numbers. We call a number \emph{fresh} at stage $s$ if it is larger than $s$ and every number seen during the construction so far; we call a set $A_i$ \emph{fresh} if $i$ is. In particular, if $A_i$ is fresh at stage $s$ then $A_{i,s}$ will be disjoint from $A_{j,s}$ for all $j \neq i$. Whenever we speak of making some $A_i$ and $A_j$ intersect, we shall mean enumerating some fresh odd number into both sets.

To motivate the proof of the theorem, we first discuss the simpler construction of an $A$ with no computable maximal subfamily with the $\overline{D}_2$ intersection property. By Remark \ref{rem_FIPsubfamilies}, it suffices to ensure, for every $e$, that $\langle A_{\Phi_e(j)}: j \in \omega \rangle$ is not a maximal subfamily. Say $\Phi_e$ \emph{enumerates} $A_i$ at stage $s$ if $\Phi_{e,s}(a) = i$ for some $a \leq s$; say it enumerates $A_i$ \emph{before} $A_j$ if $\Phi_e(a) = i$ and $\Phi_e(b) = j$ for some $a < b$. By ignoring computations if necessary, we adopt the convention that if $\Phi_e$ enumerates $A_i$ and $A_j$ at stage $s$, some number $\leq s$ belongs to $A_{i,s} \cap A_{j,s}$.

The strategy is to define a sequence of \emph{potential sets} $A_{p_{e,0}}, A_{p_{e,1}}, \ldots$ and a \emph{trap set} $A_{t_e}$ for $\Phi_e$, as follows. We wait for $\Phi_e(0)$ to converge, and if and when this happens, we define $A_{p_{e,0}}$ and $A_{t_e}$ to be the two least-indexed fresh sets. Having defined $A_{p_{e,n}}$, we intersect it with each $A_i$ enumerated by $\Phi_e$ until such a stage, if there is one, that it is itself enumerated. We call such a stage \emph{$e$-progressive}. We then let $A_{p_{e,n+1}}$ be the least-indexed fresh set. We only add elements to $A_{t_e}$ at $e$-progressive stages: if $A_{p_{e,n}}$ is the most recently enumerated potential set at such a stage, we intersect $A_{t_e}$ with all sets enumerated strictly earlier.

Now suppose $\Phi_e$ is total and defines a maximal subfamily with the $\overline{D}_2$ intersection property. Each $A_{p_{e,n}}$ is intersected with all $A_i$ enumerated by $\Phi_e$ until it is itself enumerated, which occurs at the next $e$-progressive stage. By maximality of the subfamily, then, there must be infinitely many $e$-progressive stages. Now at every such stage, we intersect $A_{t_e}$ with all the $A_i$ enumerated thus far, except for the latest $A_{p_{e,n}}$ enumerated at that stage. Thus, in the end, $A_{t_e}$ is intersected with every enumerated set, so it too must belong to the subfamily. However,  each enumerated $A_{p_{e,n}}$ is kept disjoint from $A_{t_e}$ until the enumeration of $A_{p_{e,n+1}}$, so at every stage at which $A_{t_e}$ is defined there is some enumerated $A_i$ that it does not intersect. By our convention, there is therefore no stage at which $A_{t_e}$ can be enumerated by $\Phi_e$, a contradiction.

We use the same basic idea in the proof of the theorem. The role of functionals $\Phi_e$ in enumerating members of a potential subfamily will be played by strings $\sigma \in \omega^{<\omega}$: we say $\sigma$ \emph{enumerates} $A_i$ if $\sigma(a) = i$ for some $a < |\sigma|$; we say $A_i$ is enumerated \emph{before} $A_j$ if $\sigma(a) = A_i$ and $\sigma(b) = A_j$ for some $a < b$. Our goal will be to be ensure that every maximal subfamily of $A$ with the $\overline{D}_2$ intersection property computes a function not dominated by any computable function. To this end, potential sets and trap sets will be defined for strings in a more elaborate way. In particular, we will no longer allow for some trap sets and potential sets to be undefined, as could happen above if $\Phi_e$ was not total.

\begin{proof}[Proof of Theorem~\ref{thm_FIPtoOPT}] Say a nonempty string $\sigma \in \omega^{<\omega}$ is \emph{bounded} by $s \in \omega$ if:
\begin{itemize}
\item $|\sigma| \leq s$;
\item for all $a < |\sigma|$, $\sigma(a) \leq s$;
\item if $|\sigma| > 1$ then $s > 0$ and for all $a,b < |\sigma|$, some number $\leq s-1$ belongs to $A_{\sigma(a),s-1} \cap A_{\sigma(b),s-1}$.
\end{itemize}
At each stage, we will we have defined finitely many sets $A_{p_{e,n}}$, each labeled as either a \emph{type 1 potential set} or a \emph{type 2 potential set} for some (not necessarily the same) $\sigma \in \omega^{<\omega}$. When more than one $e$ is being discussed, we refer to $A_{p_{e,n}}$ as an \emph{$e$-potential set}. For $e,a \in \omega$, let $s_{e,a} = (\mu s)[\Phi_{e,s}(a) \downarrow]$. We may assume that if $s_{e,a}$ is defined and $b < a$, then $s_{e,b}$ is defined and $s_{e,b} < s_{e,a}$.

\medskip
\noindent \emph{Construction.} At stage $s \in \omega$, we consider consecutive substages $e \leq s$. At substage $e$, we proceed as follows.

\medskip
\noindent \emph{Step 1.} If $A_{t_e}$ is undefined, define it to be the least-indexed fresh set. If $A_{t_e}$ is defined but $s = s_{e,0}$, redefine $A_{t_e}$ to be the least-indexed fresh set, and relabel any type 1 $e$-potential set as type 2 (for the same string).

\medskip
\noindent \emph{Step 2.} For each $\sigma \in \omega^{<\omega}$ bounded by $s$, choose the least $n$ such that $A_{p_{e,n}}$ is undefined, and define this set to be the least-indexed fresh set. If $\sigma$ enumerates $A_{t_e}$, label $A_{p_{e,n}}$ type 1, and otherwise label it type 2.

\medskip
\noindent \emph{Step 3.} Consider any $A_{p_{e,n}}$ defined at a stage before $s$, and any $\sigma \in \omega^{<\omega}$ bounded by $s$ that extends the string for which $A_{p_{e,n}}$ was defined as a potential set. If $A_{p_{e,n}}$ is type 1, intersect it with every $A_i$ enumerated by $\sigma$; if it is type 2, do this only if $\sigma$ does not enumerate $A_{t_e}$.

\medskip
\noindent \emph{Step 4.} Suppose $s = s_{e,a}$ for some $a$. We say that $\sigma \in \omega^{<\omega}$ is \emph{viable} for $e$ at stage $s$ if there exist $\sigma_0 \prec \cdots \prec \sigma_a = \sigma$ satisfying:
\begin{itemize}
\item $|\sigma_0| = 1$;
\item for each $b \leq a$, $\sigma_b$ is bounded by $s_{e,b}$;
\item for each $b < a$ and $c \leq b$, $\sigma_{b+1}$ enumerates a $c$-potential set for some $\tau$ with $\sigma_b \preceq \tau \prec \sigma_{b+1}$.
\end{itemize}
If $a \leq e$, we do nothing. If $a > e$, let $A_{p_{e,\sigma,a}}$ denote the least-indexed $e$-potential set enumerated by $\sigma$ for some extension of $\sigma_{a-1}$, as above. Then, if every $\sigma$ viable for $e$ enumerates a set $A_i$ satisfying:
\begin{itemize}
\item $A_i$ is enumerated before $A_{p_{e,\sigma,a}}$;
\item $A_i \cap A_{t_e}$ and $A_{p_{e,\sigma,a}} \cap A_{t_e} = \emptyset$ are currently empty;
\item $A_i$ does not equal $A_{p_{e,\tau,a}}$ for any $\tau$ viable for $e$ at stage $s$;
\end{itemize}
we choose the most recently enumerated such $A_i$, and intersect $A_{t_e}$ with it and all sets enumerated by $\sigma$ before it. In this case, we call $s$ \emph{$e$-progressive}.

\medskip
\noindent \emph{Step 5.} For each $i$ and each $n$ less than or equal to the largest number mentioned during the substage, if $n$ was not enumerated into $A_i$ we enumerate it into the complement.

\medskip
\noindent \emph{End construction.}

\medskip
\noindent \emph{Verification.} The family $A$ is computable and nontrivial, and it is not difficult to see that $A_{p_{e,n}}$ is defined for every $e$ and $n$. Likewise, $A_{t_e}$ is defined for every $e$ and is thereafter redefined at most once. We shall use $A_{t_e}$ henceforth to always refer to the final definition.

Suppose $B = \langle B_i : i \in \N \rangle$ is a maximal subfamily of $A$ with the $\overline{D}_2$ intersection property. Choose the unique $J \in \omega^\omega$ such that $B_i = A_{J(i)}$ for all~$i$.

\begin{claim}\label{claim_FIPtoHYP0}
For each $e \in \omega$ and each $\sigma \prec J$, there is an $A_{p_{e,n}} \in B$ that is a potential set for some $\tau$ with $\sigma \preceq \tau \prec J$.
\end{claim}

\begin{proof} 
Define $\tau$ as follows. If $A_{t_e}$ is not in $B$, or if it is enumerated by $\sigma$, let $\tau = \sigma$. Otherwise, let $\tau$ be an initial segment of $J$ extending $\sigma$ long enough to enumerate $A_{t_e}$. Since the $A_i$ enumerated by $\tau$ intersect pairwise, $\tau$ must be bounded by cofinitely many stages $s$. Thus, at the next such $s \geq e$, some $A_{p_{e,n}}$ is defined as a potential set for $\tau$. Fix such an $A_{p_{e,n}}$, and choose any $\upsilon$ with $\tau \preceq \upsilon \prec J$. At any future stage that bounds $\upsilon$, $A_{p_{e,n}}$ is made to intersect every set enumerated by $\upsilon$. This is so even if $A_{p_{e,n}}$ is a type 2 potential set, because in that case $A_{t_e}$ is not in $B$ and hence it is not enumerated by $\upsilon$. Since $\upsilon$ here is arbitrary, it follows that $A_{p_{e,n}}$ intersects every set in $B$, and hence by maximality that it is in $B$.\end{proof}

We now define a function $f \colon \N \to \N$, and a sequence $\sigma_0 \prec \sigma_1 \prec \cdots$ of initial segments of $J$, as follows. Let $\sigma_0 = J \res 1$ and $f(0) = 2J(0)$. Having defined $f(a)$ and $\sigma_a$ for some $a$, let $f(a+1)$ be the least $s$ such that there is a $\sigma \in \omega^{<\omega}$ satisfying:
\begin{itemize}
\item $\sigma_a \prec \sigma \prec J$;
\item $\sigma$ is bounded by $s$;
\item for each $b \leq a$, $\sigma$ enumerates a $b$-potential set defined by stage $s$ of the construction for some $\tau$ with $\sigma_a \preceq \tau \prec \sigma$.
\end{itemize} 
Let $\sigma_{a+1}$ be the least $\sigma$ satisfying the above conditions. By Claim \ref{claim_FIPtoHYP0}, $f(a)$ and $\sigma_a$ are defined for all~$a$, and it is easy to see that $\sigma_a$ is viable for $e$ at stage $s_{e,a}$.

Clearly, $f \leq_T B$, and we now show that it is not computably dominated. Seeking a contradiction, suppose $f \leq \Phi_e$. By standard conventions, we may assume $\Phi_e(a) \leq s_{e,a}$ for all $a$, and hence that $f(a) \leq s_{e,a}$.

\begin{claim}\label{claim_FIPtoHYP1} 
If $\sigma$ is viable for $e$ at stage $s_{e,a}$, then it enumerates some set that is not intersected with $A_{t_e}$ before the first $e$-progressive stage after $s_{e,a}$.
\end{claim}

\begin{proof}
First, notice that we only intersect sets at step 3 or step 4 of the construction, and, when doing so at some substage $i$, one of the sets being intersected is always either $A_{t_i}$ or $A_{p_{i,n}}$ for some $n$. Since potential sets and trap sets are always defined to be fresh, we cannot have an $e$-potential set equal to an $i$-potential set if $i \neq e$, or an $e$-potential set equal to $A_{t_i}$, or $A_{t_e}$ equal to $A_{t_i}$. Moreover, if $A_{p_{e,n}}$ is a type 2 potential set, it can only be intersected with $A_{t_e}$ at step 4 of substage $e$ of an $e$-progressive stage. 

We now proceed by induction on $a$. If $a = 0$, viability means that $\sigma$ has length $1$ and that it is bounded by $s_{e,0}$. Thus, $\sigma$ only enumerates one set, $A_{\sigma(0)}$. At stage $s_{e,0}$, $A_{t_e}$ is redefined to be fresh, and so it can only be intersected with $A_{\sigma(0)}$ at some stage $s \geq s_{e,0}$. We may assume without loss of generality that there is no $\sigma'$ viable for $e$ at $s_{e,0}$ such that $A_{\sigma'(0)}$ is intersected with $A_{t_e}$ before stage $s$.

If $A_{\sigma(0)}$ and $A_{t_e}$ are intersected at step 3 of some substage $i$ of $s$, then it must be that $A_{\sigma(0)} = A_{p_{i,n}}$ for some $n$, and that $A_{t_e}$ is enumerated by some $\tau$ bounded by $s$ extending the string $\rho$ for which $A_{p_{i,n}}$ is a potential set. Now since $A_{p_{i,n}} = A_{\sigma(0)}$ is bounded by $s_{e,0}$, it must have been defined as a potential set for $\rho$ at a stage $\leq s_{e,0}$. At that stage, $\rho$ must have been bounded, and so it must also be bounded at $s_{e,0}$. This means that $A_{\rho(n)}$ cannot equal $A_{t_e}$ for any $n < |\rho|$, again because $A_{t_e}$ is redefined at $s_{e,0}$. It also means that each $\rho(n)$ is viable for $e$ at $s_{e,0}$, so by our assumption, $A_{\rho(n)}$ cannot be intersected with $A_{t_e}$ before stage $s$. But then $\tau$ cannot be bounded by $s$, a contradiction.

Thus, if $A_{\sigma(0)}$ and $A_{t_e}$ are intersected, it is at step 4 of some substage $i$. If $i \neq e$, then $A_{\sigma(0)}$ must be $A_{t_i}$, and $A_{t_e}$ must be enumerated by some $\tau$ viable for $i$ at stage $s$. Also, $s$ must be $i$-progressive, so $s_{i,0}$ must be defined. Since $A_{t_i}$ is redefined at stage $s_{i,0}$ and $A_{\sigma(0)}$ is bounded by $s_{e,0}$, it follows that $s_{i,0} \leq s_{e,0}$. As $\tau(0)$ is bounded by $s_{i,0}$ by definition of viability, it must thus also be bounded by $s_{e,0}$. As above, this means that $A_{\tau(0)}$ cannot equal $A_{t_e}$, and that therefore the two sets cannot be intersected before stage $s$. This again contradicts that $\tau$ is bounded by $s$, and we conclude that $i = e$. In other words, $s$ is $e$-progressive, as desired.

Now take $a > 0$ and assume the claim holds for $a-1$, and suppose $\sigma$ is viable for $e$ at stage $s_{e,a}$. If $s_{e,a}$ is not $e$-progressive, then the first $e$-progressive stage after $s_{e,a}$ is the same as the first $e$-progressive stage after $s_{e,a-1}$. In this case, then, the claim follows just from the fact that some initial segment of $\sigma$ is viable for $e$ at stage $s_{e,a-1}$. If $s_{e,a}$ is $e$-progressive, then consider the set $A_{p_{e,\sigma,a}}$ enumerated by $\sigma$. This is by definition an $e$-potential set for some initial segment of $\sigma$ viable at stage $s_{e,a-1}$. As this initial segment is bounded by $s_{e,a-1}$ and, by inductive hypothesis, enumerates a set that does not intersect $A_{t_e}$ at that stage, it cannot enumerate $A_{t_e}$. Hence, $A_{p_{e,\sigma,a}}$ is of type 2. By definition of $A_{p_{e,\sigma,a}}$ is not intersected with $A_{t_e}$ at step 4 of substage $e$ of stage $s_{e,a}$, so by the remark above, this intersection can only take place at the next $e$-progressive stage.
\end{proof}

\begin{claim}\label{claim_FIPtoHYP2}
There are infinitely many $e$-progressive stages.
\end{claim}

\begin{proof}
Fix any stage $s_{e,a}$ with $a > e$, and assume there is no later $e$-progressive stage. For each $\sigma$ viable for $e$ at $s_{e,a}$, let $A_{i_\sigma}$ be the greatest-indexed set satisfying the conclusion of the statement of the preceding claim. Thus, $A_{i_{\sigma}}$ can never be intersected with $A_{t_e}$. Now for each $b > a$, and each $\sigma$ viable for $e$ at stage $s_{e,b}$, $A_{p_{e,\sigma,b}}$ is a potential set for some $\tau \prec \sigma$ viable for $e$ at stage $s_{e,b-1}$. Furthermore, since potential sets are always defined to be fresh, it follows that $A_{p_{e,\sigma,b}}$ is defined strictly later than $A_{p_{e,\tau,b-1}}$ in the course of the construction. We conclude that if $b$ is sufficiently large, then $A_{p_{e,\sigma,b}}$ does not equal any of the sets enumerated by the strings viable for $e$ at stage $s_{e,a}$. As observed above, each such $A_{p_{e,\sigma,b}}$ is a type 2 potential set, and so it can only be intersected with $A_{t_e}$ at an $e$-progressive stage. Thus, if $b$ is also chosen large enough that each $A_{p_{e,\sigma,b}}$ is defined after stage $s_{e,a}$, then $A_{p_{e,\sigma,b}}$ and $A_{t_e}$ will be forever disjoint. This implies that $s_{e,b}$ is $e$-progressive, a contradiction.
\end{proof}

We now complete the proof of the theorem as follows. First note that $A_{t_e} \notin B$. Otherwise, it would have to be enumerated by $\sigma_a$ for some $a$. But $\sigma_a$ is viable, and hence bounded, at stage $s_{e,a}$, so all sets it enumerates would need to intersect $A_{t_e}$ at stage $s_{e,a}$, contrary to Claim~\ref{claim_FIPtoHYP1}. Now consider any $e$-progressive stage $s_{e,a}$. By the construction at step 4 of substage $e$, there is some $A_i$ enumerated by $\sigma_a$ that is disjoint from $A_{t_e}$ at the beginning of the stage, and that, along with all sets enumerated by $\sigma_a$ before it, is intersected with $A_{t_e}$ by the end of the stage. Since there are infinitely many $e$-progressive stages by Claim~\ref{claim_FIPtoHYP2}, and since $J = \bigcup_a \sigma_a$, it follows that $A_{J(a)}$ intersects $A_{t_e}$ for all~$a$. This contradicts the maximality of~$B$.
\end{proof}

\section{Relationships with other principles}\label{sec_FIPandotherprinc}

By the preceding results, $F\IP$ and the principles $\overline{D}_n\IP$ are of the irregular variety that do not admit reversals to any of the main subsystems of $\mathsf{Z}_2$. Many principles of this kind have been studied in the literature, and collectively they form a rich and complicated structure. (A partial summary is given by Hirschfeldt and Shore~\cite[p.~199]{HS-2007}, with additional discussions by Montalb\'{a}n~\cite[Section 1]{Montalban-TA} and Shore~\cite{Shore-2010}.) In this section, we show that the intersection principles lie near the bottom of this structure.

Theorem \ref{thm_FIPtoOPT} gives us a lower bound on the strength of $\overline{D}_2\IP$. Examining the proof, we note that the construction there is computable, and that showing that the function $f$ defined in the verification is total and not computably dominated requires only $\Sigma^0_1$ induction. (See \cite[Definition VII.1.4]{Simpson-2009} for the formalizations of Turing reducibility and equivalence in $\RCA$.) We thus obtain the following:

\begin{corollary}\label{cor_FIPtoOPT}
Over $\RCA$, $\overline{D}_2\IP$ implies the principle $\HYP$, which asserts that for every $S$, there is a set of degree hyperimmune relative to $S$.
\end{corollary}

The reverse mathematical strength of $\HYP$ was examined by Hirschfeldt, Shore, and Slaman \cite{HSS-2009} in their investigation of certain model-theoretic principles related to the atomic model theorem $(\AMT)$. Specifically, they showed \cite[Theorem 5.7]{HSS-2009} that $\HYP$ is equivalent to the omitting partial types principle $(\OPT)$, a weaker form of $\AMT$ asserting that every complete, consistent theory has a model omitting the nonprincipal members of a given set of partial types. (See \cite[pp.~5808,~5831]{HSS-2009} for complete definitions, and \cite[Section II.8]{Simpson-2009} for a general development of model theory in $\RCA$.)

Thus, Corollary \ref{cor_FIPtoOPT} provides a connection between model-theoretic principles on the one hand, and set-theoretic principles, namely the intersection principles, on the other. We can extend this to an even firmer relationship. The following principle was introduced by Hirschfeldt, Shore, and Slaman~\cite[p.~5823]{HSS-2009}. They showed that it strictly implies $\AMT$ over $\RCA$, but that $\AMT$ implies it over $\RCA + \I\Sigma^0_2$ (see \cite{HSS-2009}, Theorem~4.3, Corollary~4.5, and p.~5826).

\begin{principle}[$\Pi^0_1$ genericity principle $(\Pi^0_1\mathsf{G})$]
For any uniformly $\Pi^0_1$ collection of sets $D_i$, each of which is dense in $2^{<\N}$, there is a set $G$ such that for every $i$, $G \res s \in D_i$ for some~$s$.
\end{principle}

\begin{proposition}\label{prop_Pi01GtoFIP} 
$\Pi^0_1\mathsf{G}$ implies $F\IP$ over $\RCA$.
\end{proposition}

\begin{proof} 
We argue in $\RCA$. Let a nontrivial family $A = \langle A_i : i \in \N \rangle$ be given. We may assume $A$ has no finite maximal subfamily with the $F$ intersection property. Fix a bijection $c \colon \N \to \N^{<\N}$. Given $\sigma \in 2^{<\N}$, we say that a number $n < |\sigma|$ is \textit{acceptable for} $\sigma$ if:
\begin{itemize}
\item $\sigma(n) = 1$;
\item $c(n) = \tau b$, which we call the \emph{witness} of $x$, where:
\begin{itemize}
\item $\tau \in \N^{<\N}$,
\item $b \in \N$,
\item and some number $\leq b$ belongs to $\bigcap_{i < |\tau|} A_{\tau(i)}$.
\end{itemize}
\end{itemize} 
We define the \emph{{acceptable} sequence} of $\sigma$ to be either the empty string if there is no acceptable number for $\sigma$, or else the longest sequence $n_0 \cdots n_k \in \N^{<\N}$, $k \geq 0$, such that:
\begin{itemize}
\item $n_0$ is the least acceptable number for $\sigma$;
\item each $n_i$ is acceptable, say with witness $\tau_i b_i$;
\item for each $i < k$, $n_{i+1}$ is the least acceptable $n > n_i$ such that if $\tau b$ is its witness then $\tau \succ \tau_i$.
\end{itemize}
Note that $\Sigma^0_0$ comprehension suffices to prove the existence of a function $2^{<\N} \to \N^{<\N}$ which assigns to each $\sigma \in 2^{<\N}$ its acceptable sequence.

Now for each $i \in \N$, let $D_i$ be the set of all $\sigma \in 2^{<\N}$ that have a nonempty acceptable sequence $n_0\cdots n_k$, and if $\tau b$ is the witness of $n_k$ then:
\begin{itemize}
\item either $\tau(m) = i$ for some $m < |\tau|$,
\item or $A_i \cap \bigcap_{m < |\tau|} A_{\tau(m)} = \emptyset$.
\end{itemize} 
The $D_i$ are clearly uniformly $\Pi^0_1$, and it is not difficult to see that they are dense in $2^{<\N}$.
%Indeed, let $\sigma \in 2^{<\N}$ be given, and define $b$, $j$, and $x$ as follows: if the acceptable sequence of $\sigma$ is empty, choose the least $j \geq i$ such that $A_j \neq \emptyset$ and let $b \geq \min A_j$ be large enough that $x = c^{-1}(jb) \geq |\sigma|$; if the acceptable sequence of $\sigma$ is some nonempty string $x_0 \cdots x_n$ and $\tau b_n$ is the witness of $x_n$, choose the least $j \geq i$ such that $A_j \cap \bigcap_{k < |\tau|} A_{\tau(k)} \neq \emptyset$ and let $b \geq \min A_j \cap \bigcap_{k < |\tau|} A_{\tau(k)}$ be large enough that $x = c^{-1}(\tau jb) \geq |\sigma|$. In either case, $j$ exists because of our assumption that $A$ is nontrivial and has no finite maximal subfamily with the $F$ intersection property. Now define $\widehat{\sigma} \in 2^{<\N}$ of length $x+1$ by
%\[
%\widehat{\sigma}(y) =
%\begin{cases} 
%\sigma(y) & \text{if } y < |\sigma|,\\ 
%0 & \text{if }|\sigma| \leq y < x,\\
%1 & \text{if } y = x,
%\end{cases}
%\]
%to get an extension of $\sigma$ that belongs to $D_i$.
Apply $\Pi^0_1\mathsf{G}$ to the $D_i$ to obtain a set $G$ such that for all $i$, there is an $s$ with $G \res s \in D_i$. By definition, each such $s$ must be nonzero as $G \res s$ must have nonempty acceptable sequence. Note also that if $s \leq t$ then the acceptable sequence of $G \res t$ extends (not necessarily properly) the acceptable sequence of $G \res s$. Our assumption that $A$ has no finite maximal subfamily with the $F$ intersection property implies that the acceptable sequences of initial segments of $G$ are arbitrarily long.

Now fix the least $s$ such that $G \res s$ has a nonempty acceptable sequence, and for each $t \geq s$, if $n_0\cdots n_k$ is the acceptable sequence of $G \res t$, let $\tau_t b_t$ be the witness of $n_k$. By the preceding paragraph, we have $\tau_t \preceq \tau_{t+1}$ for all $t$, and $\lim_t |\tau_t| = \infty$. Let $J = \bigcup_{t \geq s} \tau_t$, which exists by $\Sigma^0_0$ comprehension. It is then straightforward to check that $B = \langle A_{J(i)}: i \in \N \rangle$ is a maximal subfamily of $A$ with the $F$ intersection property.
\end{proof}

By Corollary 3.9 of \cite{HSS-2009}, there is an $\omega$-model of $\AMT$, and hence of $\Pi^0_1\mathsf{G}$, that is not a model of $\WKL$. Hence, $F\IP$ does not imply $\WKL$, and so in view of Corollary \ref{cor_WKLnottoFIP} the two are incomparable. $F\IP$ also inherits from $\Pi^0_1\mathsf{G}$ conservatively for restricted $\Pi^1_2$ sentences, i.e., those of the form $(\forall X)[\varphi(X) \to (\exists Y)\psi(X,Y)]$, where $\varphi$ is arithmetical and $\psi$ is $\Sigma^0_3$. (This fact can also be established directly, by replacing the forcing notion in the proofs of Proposition 3.14 and Corollary 3.15 of \cite{HSS-2009} by the notion defined in Proposition \ref{prop_FIPnotrevACA0} above.) It follows, for example, that $F\IP$ does not imply any of the combinatorial principles related to Ramsey's theorem for pairs studied by Cholak, Jockusch, and Slaman \cite{CJS-2001} or Hirschfeldt and Shore \cite{HS-2007}.

We do not know whether the preceding proposition can be strengthened to show that $F\IP$ follows from $\AMT$ over $\RCA$. We also do not know whether $\HYP$ implies $F\IP$, although the following proposition and theorem provide partial steps in this direction.

\begin{proposition}
Let $A = \langle A_i : i \in \N \rangle$ be a computable nontrivial family of sets. Every set $D$ of degree hyperimmune relative to\/ $\0'$ computes a maximal subfamily of $A$ with the $F$ intersection property.
\end{proposition}

\begin{proof}
We may assume that $A$ has no finite maximal subfamily with the $F$ intersection property. By deleting some of the members of $A$ if necessary, we may further assume that $A_0 \neq \emptyset$. Define a $\emptyset'$-computable function $g \colon \N \to \N$ by letting $g(s)$ be the least $n$ such that for all finite sets $F \subseteq \{0,\ldots,s\}$,
\[
\bigcap_{j \in F} A_j \neq \emptyset \Longrightarrow (\exists a \leq n)[x \in \bigcap_{j \in F} A_j].
\]
Since $D$ has hyperimmune degree relative to $\0'$, we may fix a function $f \leq_T D$ not dominated by any $\emptyset'$-computable function. In particular, $f$ is not dominated by~$g$.

Now define $J \colon \omega \to \omega$ inductively as follows: let $J(0) = 0$, and having defined $J(s)$ for some $s \geq 0$, search for the least $i \leq s$ not yet in the range of $J$ for which there is an $a \leq f(s)$ with
\[
a \in A_i \cap \bigcap_{j \leq s} A_{J(j)},
\]
setting $J(s+1) = i$ if it exists, and setting $J(s+1) = 0$ otherwise.

Clearly, $J \leq_T f$. Moreover, $\bigcap_{i \leq s} A_{J(i)} \neq \emptyset$ for every $s$, so the subfamily defined by $J$ has the $F$ intersection property. We claim that for all $i$, if $A_i \cap \bigcap_{j \leq s} A_{J(j)} \neq \emptyset$ for every $s$ then $i$ is in the range of~$J$. Suppose not, and let $i$ be the least witness to this fact. Since $f$ is not dominated by $g$, there is an $s \geq i$ such that $f(s) \geq g(s)$ and for all $t \geq s$, $J(t) \neq j$ for any nonzero $j < i$. By construction, $J(j) \leq j$ for all $j$, so the set $F = \{i\} \cup \{J(j) : j \leq s\}$ is contained in $\{0,\ldots,s\}$. Consequently, there necessarily is some $a \leq g(s)$ with $a \in A_i \cap \bigcap_{j \leq s} A_{J(j)}$. But then $x \leq f(s)$, so $J(s+1)$ is defined to be $i$, which is a contradiction. We conclude that $\langle A_{J(i)} : i \in \omega \rangle$ is maximal, as desired.
\end{proof}

\begin{theorem}\label{thm_FIPpermitting}
Let $A = \langle A_i : i \in \N \rangle$ be a computable nontrivial family of sets. Every noncomputable c.e.\ set $W_e$ computes a maximal subfamily of $A$ with the $F$ intersection property.
\end{theorem}

\begin{proof}
As above, assume that $A$ has no finite maximal subfamily with the $F$ intersection property, and that $A_0 \neq \emptyset$. We build a limit computable set $M$ by a permitting argument. Let $M_s$ denote our approximation to $M$ at stage $s$ of the construction. Fix a computable approximation $\langle W_{e,s} : s \in \omega \rangle$ of $W_e$, and assume this has been advanced if necessary to ensure that $W_{s+1} - W_s \neq \emptyset$ for all $s$. 

\medskip
\noindent \emph{Construction.} For each $i$ and each $n$, call $\langle i,n \rangle$ a \emph{copy} of~$i$.

\medskip
\noindent \emph{Stage~$0$.} Enumerate $\langle 0, 0 \rangle$ into
$M_0$.

\medskip
\noindent \emph{Stage $s+1$.} Assume that $M_s$ has been defined, that it is finite and contains $\langle 0, 0 \rangle$, and that each $i$ has at most one copy in $M_s$. For each $i$ with no copy in $M_s$, let $\ell(i,s)$ be the greatest $k$ with a copy in $M_s$, if it exists, such that some number $\leq s$ belongs to $A_i$ and to $A_j$ for every $j \leq k$ with a copy in $M_s$.

Now consider all $i \leq s$ such that the following hold:
\begin{itemize}
\item $\ell(i,s)$ is defined;
\item there is no $j$ with a copy in $M_s$ such that $\ell(i,s) < j < i$;
\item for all $\langle j,n \rangle \in M_s$, if $\ell(i,s) < j$ then $W_{e,s} \res \langle j,n \rangle \neq W_{e,s+1} \res \langle j,n \rangle$.
\end{itemize}
If there is no such $i$, let $M_{s+1} = M_{t_{s+1}}$. Otherwise, fix the least such $i$, and let $M_{s+1}$ be the result of removing from $M_s$ all $\langle j,n \rangle$ for $j > \ell(i,s)$, and enumerating in the least copy of $i$ greater than or equal to the least element of $W_{e,s+1} - W_{e,s}$, and greater than every element of $M_s$.

\medskip
\noindent \emph{End construction.}

\medskip
\noindent \emph{Verification.} The construction ensures that for all $m$ and $s$, if $M_{s}(m) \neq M_{s+1}(m)$ then $W_{e,s} \res m \neq W_{e,s+1} \res m$. It follows that $M(m) = \lim_s M_s(m)$ exists for all $m$ and is computable from $W_e$. Furthermore, note that $\bigcap_{\langle i,n \rangle \in M_s} A_i \neq \emptyset$ for all~$s$. Thus, if $F$ is any finite subset of $M$, then $\bigcap_{\langle i, n \rangle \in F} A_i \neq \emptyset$, because $F$ is necessarily a subset of $M_s$ for some~$s$. If we now let $J \colon \omega \to \omega$ be any $W_e$-computable function with range equal to $\{ i : (\exists n)[\langle i,n \rangle \in M]\}$, it follows that $\langle A_{J(i)} : i \in \omega \rangle$ has the $F$ intersection property.

We claim that this subfamily is also maximal. Seeking a contradiction, suppose not, and let $i$ be the least witness to this fact. So $A_i \cap \bigcap_{\langle j,n\rangle \in F} A_j \neq \emptyset$ for every finite subset $F$ of $M$, but no copy of $i$ belongs to~$M$. By construction, $\langle 0,0\rangle \in M_s$ for all $s$ and hence also to $M$, so it must be that $i > 0$. Let $i_0 < \cdots < i_r$ be the numbers less than $i$ that have copies in $M$, and let these copies be $\langle i_0,n_0 \rangle, \ldots, \langle i_r, n_r \rangle$, respectively. Let $s$ be large enough so that:
\begin{itemize}
\item some number $\leq s$ belongs to $A_i \cap \bigcap_{j \leq r} A_{i_j}$;
\item for all $t \geq s$ and all $j \leq r$, $\langle i_j, n_j \rangle \in M_t$.
\end{itemize} 
Then no copy of any $j \leq i$ not among $i_0,\ldots,i_r$ can be in $M_t$ at any stage $t \geq s$. For if it were, it would have to be removed at some later stage, which could only be done for the sake of enumerating a copy of some $k < j$. This $k$, in turn, could not be any of $i_0,\ldots,i_r$ by choice of $s$, and so it too would subsequently have to be removed. Continuing in this way would result in an infinite regress, which is impossible.

It follows that for all $t \geq s$, $\ell(i,t)$ is defined and no smaller than $i_r$, and its value tends to infinity. Since no copy of $i$ is enumerated at any such $t$, there must be some $j > \ell(i,t)$ with a copy $\langle j,n \rangle$ in $M_t$. Let $\langle j_t,n_t \rangle$ be the least such copy at stage~$t$. Then $\langle j_t,n_t \rangle \leq \langle j_{t+1},n_{t+1} \rangle$ for all $t$, because a number is enumerated into $M$ at a given stage only if it is larger than all numbers that were in $M$ at the previous one. Furthermore, for infinitely many $t$ this inequality must be strict, since infinitely often $\ell(i,t+1) \geq j_t$.

Now fix any $t \geq s$ so that $\ell(i,u) \geq i$ for all $u \geq t$. We claim that for all $v \geq u \geq t$, $W_{e,u} \res \langle j_u,n_u \rangle = W_{e,v} \res \langle j_u,n_u \rangle$. If not, choose the least $v \geq u$ such that $W_{e,v} \res \langle j_u,n_u \rangle \neq W_{e,v+1} \res \langle j_u,n_u \rangle$. Then $W_{e,v} \res \langle j_v,n_v \rangle \neq W_{e,v+1} \res \langle j_v,n_v \rangle$, so $D_v \res \langle j,n \rangle \neq W_{e,v+1} \res \langle j,n \rangle$ for every $\langle j,n \rangle \in M_v$ with $j > \ell(i,v)$. But this means that some copy of $i$ is enumerated into $M_{v+1}$, a contradiction. Thus, we have that $W_{e,u} \res \langle j_u,n_u \rangle = W_e \res \langle j_u,n_u \rangle$, so given any $n$, we can compute $W_e \res n$ simply by searching for a $u \geq t$ with $\langle j_u, n_u \rangle \geq n$. This contradicts the assumption that $W_e$ is noncomputable, completing the proof.
\end{proof}

A first attempt at showing that $\HYP$ implies $F\IP$ might be the following. Given a family $A = \sequence{A_i : i \in \omega}$ and function $f$ that is not computably dominated, define the subfamily by putting $A_i$ in at stage $s$ if $i$ is least such that $f(s)$ bounds a witness for the intersection of $A_i$ with all members of $A$ put in so far. Then, for each $i$, define a function $g_i$ by letting $g_i(s)$ be so large that it bounds such a witness for $A_i$. This way, if $A_i$ intersects all members of our subfamily, $g_i$ must be total, and so $A_i$ must eventually be put in provided there are infinitely many $s$ such that $f(s) \geq g_i(s)$. By choice of $f$, the latter condition holds if $g_i$ is computable, but in general it needs only to be computable in our approximation to the subfamily. One way to think of the preceding theorem is as saying that if $f$ has c.e.\ degree then we can make this approximation, and hence $g_i$, computable.

Our final result shows that $F\IP$ does not imply $\Pi^0_1\mathsf{G}$, or even $\AMT$.

\begin{corollary}
Over $\RCA$, $F\IP$ does not imply $\AMT$.
\end{corollary}

\begin{proof}
Csima, Hirschfeldt, Knight, and Soare~\cite[Theorem~1.5]{CHKS-2004} showed that for every set ${D} \leq_T \emptyset'$, if every complete atomic decidable theory has an atomic model computable from $D$, then $D$ is nonlow$_2$. Thus $\AMT$ cannot hold in any $\omega$-model all of whose sets have degree bounded by a fixed low$_2$ $\Delta^0_2$ degree. By contrast, using Theorem \ref{thm_FIPpermitting}, we can build such a model of $F\IP$: for example, take any sequence of the form $\emptyset <_T S_0 <_T S_1 <_T \cdots <_T W_e$ where $W_e$ is a low$_2$ c.e.\ set, and take the model of all sets computable in some $S_i$.
\end{proof}

%\begin{figure}
%\[
%\xymatrix@R-10pt@C-10pt{
%\ACA \ar@2[dddd] \ar@2[dr] \ar@1[r] & D_n\IP \ar@1[l]\\
%& \Pi^0_1\mathsf{G} \ar@/_0pc/[dddl] |-{\object@{|}} |>{\object@{}} \ar@2[ddd] \ar@2[dr]\\
%& & F\IP \ar@/_0pc/[ddl] |-{\object@{|}} |>{\object@{}} \ar@1[d] \\
%& & \vdots \ar@1[d] \\
%\WKL \ar@/_0pc/[dddr] |-{\object@{|}} |>{\object@{}} \ar@2[dddd] & \AMT \ar@2[ddd] & \overline{D}_n \IP \ar@1[d] \\
%& & \vdots \ar@1[d] \\
%& & \overline{D}_2 \IP \ar@1[dl] \\ & \OPT \ar@2[dl] \\ \RCA\
%}
%\]
%\caption{Position of the intersection principles below $\ACA$, with $n \geq 2$ arbitrary. Arrows denote implications in $\RCA$, double arrows strict implications, and negated arrows non-implications.}\label{fig_summary}
%\end{figure}

We conclude by stating the questions left open by our investigation. We conjecture the answer to part (3) to be no.

%Our results are summarized in Figure \ref{fig_summary}.
\begin{question}
\
\begin{enumerate}
\item For any $n \in \omega$, does $\overline{D}_n\IP$ imply $F\IP$ or at least $\overline{D}_{n+1}\IP$?
\item Does $\AMT$ imply $\overline{D}_2\IP$ over $\RCA$? Does $\OPT$ imply $\overline{D}_2\IP$?
\item Does every computable nontrivial family of sets have a maximal subfamily with the $F$ intersection property computable in a given set of hyperimmune degree?
\end{enumerate}
\end{question}

\bibliographystyle{amsplain}
\bibliography{Choice}

\end{document}